\documentclass[a4paper, 11pt, twoside]{article}
\usepackage{amsthm}
\usepackage{amsmath}
\usepackage{fancyhdr}
\usepackage{amssymb}
\usepackage{mathtools}
\usepackage{indentfirst}
\usepackage{color}
\usepackage{bm}
\usepackage{extarrows}
\usepackage[pagewise]{lineno}
\usepackage{babel}
\usepackage{hyperref}
\usepackage[utf8]{inputenc}
\usepackage{amsfonts}
\usepackage{lmodern}
\usepackage{tikz}
\usepackage{enumerate}
\usetikzlibrary{matrix}
\usepackage{mathrsfs}
\usepackage{etoolbox}
\usepackage{keyval}
\usepackage{csquotes}
\usepackage{mathtools}
\usepackage{dsfont}
\usepackage{tcolorbox}
\usepackage{bbm}
\usepackage{url}

\theoremstyle{plain}
\newtheorem{proposition}{Proposition}
\newtheorem{lemma}[proposition]{Lemma}
\newtheorem{theorem}[proposition]{Theorem}

\theoremstyle{definition}

\theoremstyle{definition}
\newtheorem{conj}{Conjecture}
\newtheorem{rem}[conj]{Remark}
\theoremstyle{plain}

\numberwithin{equation}{section}
\numberwithin{proposition}{section}
\numberwithin{conj}{section}	

\usepackage{amsmath}

\usepackage{geometry}
\begin{document}

\title{Optimal Control for Kuramoto Model: from Many-Particle Liouville Equation to Diffusive Mean-Field Problem}
\author{
\sl{Li Chen}\\
   \small\emph{School of Business Informatics and Mathematics, University of Mannheim,} \\
   \small\emph{Mannheim, 68131, Germany}\\
\sl{Yucheng Wang\footnote{Corresponding author}}\\
 \small\emph{Department of Applied Mathematics, The Hong Kong Polytechnic University,}\\
    \small\emph{Hong Kong, China}\\
\sl{Valeriia Zhidkova}\\
   \small\emph{School of Business Informatics and Mathematics, University of Mannheim,}\\
   \small\emph{Mannheim, 68131, Germany}
    }
\date{}
\footnotetext[1]{MSC: 35Q92, 35Q93, 92B25, 35K55.}
\footnotetext[2]{Keywords: Mean field control, optimal control of PDE, relative entropy method, Kuramoto model.}
\footnotetext[3]{E-mail addresses: chen@math.uni-mannheim.de, yucheng.wang@polyu.edu.hk, valeriia.zhidkova@uni-mannheim.de.}

\maketitle

\begin{abstract}
In this paper, we investigate the mean-field optimal control problem of a swarm of Kuramoto oscillators. Using the notion of wrapped distribution, we explain the connection between the stochastic particle system and the mean-field PDE on the periodic domain. In the limit of an infinite number of oscillators the collective dynamics of the agents' density is described by a diffusive mean-field model in the form of a non-local PDE, where the non-locality arises from the synchronization mechanism. We prove the existence of the optimal control of the mean-field model by using $\Gamma$-convergence strategy of the cost functional corresponding to the Liouville equation on the particle level. In the discussion of propagation of chaos for fixed control functions we complete the relative entropy estimate by using large deviation estimate given by \cite{MR3858403}. 
\end{abstract}

\section{Introduction}

Synchronization refers to a process during which two or more dynamical systems adjust some of their properties to a common behavior due to strong or weak coupling. This phenomenon is observed in biological, chemical, physical, and social systems and it has attracted the interest of scientists for centuries. Some examples in biology and medicine include synchronous flashing of fireflies \cite{fireflies}, synchronization of oscillations of human insulin secretion and glucose infusion \cite{21fffdf5a40e4959ba645110b0d27bf8} and synchronized firing of pacemaker cells in the heart \cite{winfree}. Neuroscience research of cortical processes has much to gain from understanding the fundamental mechanisms governing fluctuating oscillations in large-scale cortical circuits. We refer to \cite{neurobiology} for a detailed overview of neurobiological implications of deeper mathematical understanding of synchronization mechanisms. The robotics community has devised control strategies for large-scale collectives of robots with the goal of mimicking the resilience, efficiency, and adaptive behavior of biological swarms \cite{robotics}. There are many engineering examples, such as synchronization in power grids \cite{powergrids} or in laser arrays \cite{BLACKBEARD201443}. Moreover, the generalization that complex systems tend to be
oscillatory has been recently applied in environment and sustainability research, for example for ecosystems of pests and crops \cite{ecosystems} and evolution of heat waves in urban climate networks \cite{WANG2021100909}. We refer to Chapter VII in \cite{acebron} for a summary of various applications of oscillator models. \\

The most successful attempt at mathematical treatment of these synchronized phenomena was due to Kuramoto \cite{kuramoto}, who analyzed a model of phase oscillators running at arbitrary intrinsic frequencies, and coupled through the sine of their phase differences. The original Kuramoto model with all-to-all coupling and random intrinsic frequencies has the following form:
\begin{align} \label{kuramoto}
\dot{\theta}_i = \omega_i + \frac{K}{N} \sum_{j=1}^N \sin (\theta_j - \theta_i).
\end{align}
Here, $\theta_i : \mathbb{R} \rightarrow S^1 := \mathbb{R} / 2\pi \mathbb{Z}, \ i = 1, \dots, N$ is the phase of the oscillator $i$, whose intrinsic frequency $\omega_i$ is drawn from the probability distribution, $N$ is the number of oscillators, and $K$ is the strength of coupling. The sum on the right-hand side describes the interactions of the oscillators in the network. The Kuramoto model has been extensively analyzed to find out the synchronization mechanism \cite{ha,verwoerd}, effect of network structures \cite{1383983} and the coupling strength which induces the synchronization \cite{CHOI2012735, HA20101692}.  \\

The classical mean-field theory has proven to be a useful tool in understanding collective behavior of multi-agent systems. For many such models the description of the asymptotic behavior of a very large number of agents can become numerically expensive. One of the attempts to tackle this difficulty is the mean-field theory which simplifies the effect of all other individuals on any given individual by a single averaged effect. We refer to \cite{Jabin2017} and \cite{chen2024fluctuationsmeanfieldlimitattractive} and the references therein for a recent summary on the mean-field limit of large ensembles of interacting particles.  \\

Optimal control problems are naturally considered within the mean-field framework because of the incompleteness of self-organization. As an example, under certain circumstances, the formation of a specific pattern is conditional to the initial datum being positioned in a corresponding basin of attraction \cite{Chuang}. Thus it is interesting to question, whether a policy maker, represented by a control function, can drive the system towards pattern formation. One of the first works to introduce the concept of mean-field optimal control was \cite{Fornasier}. We also refer to \cite{chen2024meanfieldcontroldiffusionaggregation} and references therein for a summary of the recent developments in the mean-field optimal control theory. \\

In this paper, we investigate the optimal control problem proposed in \cite{sinigaglia2021optimalcontrolvelocitynonlocal}. There, the authors Carlo Sinigaglia, Francesco Braghin, and Spring Berman consider the following controlled mean-field PDE for the Kuramoto model
\begin{align} \label{mfpde}
\partial_t q - D \partial_{\theta}^2 q + \partial_\theta \left( q \left( u_1 + u_2 w [q] \right) \right) = 0,
\end{align}
where
\begin{itemize}
	\item 
$w[q] = \displaystyle\int_{S^1} \sin(\theta ' - \theta - \alpha) q(t, \theta ') \, d\theta',$
	\item $q(t,\theta) : [0,T] \times S^1 \rightarrow [0,\infty)$ is the oscillator density at time $t$, 
	\item $D>0$ is the diffusion coefficient,
	\item $\alpha$ is a constant phase shift,
	\item $u_1: [0,T] \times S^1 \rightarrow \mathbb{R}$ is the control function of the angular velocity,
	\item $u_2: [0,T] \times S^1 \rightarrow \mathbb{R}$ is the control function that modulates the strength of agent interactions.
\end{itemize}
The term $\omega[q]$ is the non-locality that arises from the synchronization mechanism in the Kuramoto dynamics (\ref{kuramoto}). The authors formulate the optimal control problem subject to the mean-field PDE (\ref{mfpde}). A natural objective is to compute
the minimum control inputs $u_1 , u_2$ that drive the oscillator density to a target density $z(t,\theta)$ at a final time $t = T$. The cost functional given in \cite{sinigaglia2021optimalcontrolvelocitynonlocal} to be minimized is defined as
\begin{align} \label{costfunc}
\begin{split}
\mathcal{J} (q[u_1,&u_2], u_1, u_2) = \frac{\alpha_r}{2} \int_0^T \int_{S^1}  (q(t,\theta) - z(t,\theta))^2 \, d\theta dt \\
&+ \frac{\alpha_t}{2} \int_{S^1} (q(T,\theta) - z(T,\theta))^2 \, d\theta + \frac{1}{2} \int_0^T \int_{S^1} (\beta_1 u_1^2 (t,\theta) + \beta_2 u^2_2(t,\theta)) \, d\theta dt,
\end{split}
\end{align} 
where $\alpha_r, \alpha_t, \beta_1, \beta_2 >0$ are weighting constants and $q = q[u_1,u_2]$ solves (\ref{mfpde}) with given initial data $q_0$. The goal of the optimal control problem is to find functions $\overline{u}_1, \overline{u}_2$ such that
\begin{align} \label{ocp}
\mathcal{J}(q[\overline{u}_1,\overline{u}_2], \overline{u}_1,\overline{u}_2) =  \min_{u_1,u_2} \mathcal{J}(q[u_1,u_2], u_1,u_2).
\end{align}

The authors use the Lagrangian method in order to derive a set of first-order necessary optimality conditions and then proceed to solve the optimal control problem numerically using an iterative gradient descent method. Analyzing the results, the authors come to a conclusion that both the velocity and interaction strength control inputs are able to successfully drive the oscillator density to a target distribution while increasing the self-synchronization speed. Interaction strength control is able to achieve similar performance as velocity control, with slightly slower convergence to synchronization, despite its controllability limitations. \\

To our knowledge, \cite{sinigaglia2021optimalcontrolvelocitynonlocal} constitutes the first attempt at formulating an optimal control problem for mean-field dynamics of identical Kuramoto oscillators. Building on that, we aim to prove solvability of proposed control problem (\ref{ocp}) using the mean-field control framework. More precisely, we use the Liouville equation corresponding to (\ref{mfpde}) to formulate the microscopic control problem and then take the limit $N \rightarrow \infty$ using Gamma-convergence techniques. This approach does not only allow us to obtain the existence of the minimizer of (\ref{ocp}) but also helps us understand the same control problem in different scales. We discuss this approach in more detail in the following.  \\

As mentioned in \cite{sinigaglia2021optimalcontrolvelocitynonlocal}, the mean-field PDE (\ref{mfpde}) is motivated by the following controlled dynamical many-particle model:
\begin{align}\label{kuramotoparticle}
\begin{cases}
d\theta_i = u_1(t,\theta_i ) \, dt + \displaystyle \frac{u_2(t,\theta_i)}{N} \displaystyle \sum_{j=1}^N \sin (\theta_j - \theta_i - \alpha) \, dt+ \sqrt{2D} \, d W_i (t), \ t>0,   \\
\theta_i (0) = \theta_i^0, \ i=1,\dots,N,
\end{cases}
\end{align} 
where $(W_i(t))_{i=1}^N$ denote independent Brownian motions and the initial data $(\theta_i^0)_{i=1}^N$ are assumed to be i.i.d. random variables. \\
In the following, we explain the connection between McKean-Vlasov SDE system and the corresponding non-local PDE on $S^1$ as well as between the particle SDE system and its Liouville equation on $(S^1)^N$ making use of the wrapped distribution
\footnote{For a random variable $x$ on $\mathbb{R}$ with probability density $f$, the corresponding random variable $x_w := x \mod 2\pi$ has wrapped probability density $f_w$ given by
\begin{align*}
f_w(y) = \sum_{k=-\infty}^\infty f(y + 2\pi k). 
\end{align*}
For an arbitrary $2\pi$-periodic function $g:\mathbb{R} \rightarrow\mathbb{R}$ it holds
\begin{align} \label{circint}
\begin{split}
\int_\mathbb{R} g(x) f(x) \, dx = \sum_{k=-\infty}^\infty \int_{2\pi k}^{2\pi (k+1)} g(x) f(x) \, dx  
= \int_{0}^{2\pi} g(x) f_w (x) \, dx .
\end{split}
\end{align}}. We refer to \cite{dirstat} for more information about wrapped distributions. For given control functions $u_1,u_2$ on $S^1$ we use the same notation for their periodic extensions to $\mathbb{
R}$. We start by noticing that the mean-field problem of (\ref{kuramotoparticle}) (as a problem on $\mathbb{R}^N$)
\begin{align*}
\begin{cases}
d\tilde{\theta} = u_1(t,\tilde{\theta}) \, dt + u_2(t,\tilde{\theta}) \left(\displaystyle \int_\mathbb{R} \sin (\hat{\theta} - \tilde{\theta}- \alpha) \tilde{q}(\hat{\theta}) \, d\hat{\theta} \right) \, dt+ \sqrt{2D} \, d W (t), \ t>0,   \\
\tilde{\theta}(0) = \theta_0, \ \tilde{q} = \mathrm{Law} (\tilde{\theta})
\end{cases}
\end{align*}
admits a unique strong solution $\tilde{\theta} : \Omega \times [0,T] \rightarrow \mathbb{R}$. This can be easily proven by a standard fixed point argument, see for example \cite{lacker2018mean}. With (\ref{circint}) it follows that $\overline{\theta} := \tilde{\theta} \mod 2\pi$ solves the McKean-Vlasov SDE on $S^1$
\begin{align} \label{mckeanvlasov}
\begin{cases}
d\overline{\theta} = u_1(t,\overline{\theta}) \, dt + u_2(t,\overline{\theta}) \left(\displaystyle \int_0^{2\pi} \sin (\tilde{\theta} - \overline{\theta} - \alpha) q(\tilde{\theta}) \, d\tilde{\theta} \right) \, dt+ \sqrt{2D} \, d W (t), \ t>0,   \\
\overline{\theta}(0) = \theta_0 \mod 2\pi, \ q = \mathrm{Law} (\overline{\theta}).
\end{cases}
\end{align}
A given test function $\phi : S^1 \rightarrow \mathbb{R}$ is a $2\pi$-periodic function on $\mathbb{R}$. This is then also true for its partial derivatives. This fact allows us to apply the Itô formula to $\phi(\tilde{\theta})$ and, after taking expectation, to rewrite the resulting equation in terms of expectations on $S^1$ using the notion of wrapped distribution.
 We obtain
\begin{align*}
\phi(\tilde{\theta}(t)) - \phi(\tilde{\theta}(0)) = &\int_0^t \phi' (\tilde{\theta}(s)) \left( u_1(s,\tilde{\theta}(s)) +u_2(s,\tilde{\theta}(s)) \left( \int_\mathbb{R} \sin (\hat{\theta} - \tilde{\theta} -\alpha) \tilde{q}(\hat{\theta}) \, d\hat{\theta} \right) \right) \, ds \\
&+ \int_0^t \sqrt{2D}\phi' (\tilde{\theta} (s)) \, dW(s)+ D \int_0^t \phi '' (\tilde{\theta}(s)) \, ds
\end{align*}
and
\begin{align*}
&\int_\mathbb{R} \phi(\tilde{\theta}) \tilde{q} (t,\tilde{\theta}) \, d\tilde{\theta} - \int_\mathbb{R} \phi(\tilde{\theta}) \tilde{q} (0,\tilde{\theta}) \, d\tilde{\theta}\\ 
= &\int_\mathbb{R} \int_0^t \phi' (\tilde{\theta}) \left( u_1(s,\tilde{\theta}) +u_2(s,\tilde{\theta}) \left( \int_\mathbb{R} \sin (\hat{\theta} - \tilde{\theta} -\alpha) \tilde{q}(\hat{\theta}) \, d\hat{\theta} \right) \right) \tilde{q}(s,\tilde{\theta}) \, d\tilde{\theta}\, ds \\
&+ D \int_\mathbb{R} \int_0^t \phi '' (\tilde{\theta}) \tilde{q} (s, \tilde{\theta}) \, d\tilde{\theta} \, ds.
\end{align*}
Now using (\ref{circint}) and the $2\pi$-periodicity of $\phi$ and its derivatives we obtain
\begin{align*}
&\int_0^{2\pi} \phi(\overline{\theta}) q (t,\overline{\theta}) \, d\overline{\theta} - \int_0^{2\pi} \phi(\overline{\theta}) q (0,\overline{\theta}) \, d\overline{\theta} \\
= &\int_0^{2\pi}\int_0^t \phi' (\overline{\theta}) \left( u_1(s,\overline{\theta}) +u_2(s,\overline{\theta}) \left( \int_0^{2\pi} \sin (\hat{\theta} - \overline{\theta} -\alpha) q(\hat{\theta}) \, d\hat{\theta} \right) \right) q(s,\overline{\theta}) \, d\overline{\theta}\, ds \\
&+ D \int_0^{2\pi}
 \int_0^t \phi '' (\overline{\theta})q (s, \overline{\theta}) \, d\overline{\theta} \, ds.
\end{align*}
This shows that the wrapped distribution $q = \mathrm{Law} (\overline{\theta})$ from (\ref{mckeanvlasov}) solves (\ref{mfpde}).

Using this argument, we derive the Liouville equation corresponding to (\ref{kuramotoparticle}). The system (\ref{kuramotoparticle}) has a unique strong solution $(\tilde{\theta}_1, \dots, \tilde{\theta}_N) : \Omega \times [0,T] \rightarrow \mathbb{R}^N$ due to the standard SDE theory. Considering this solution "wrapped" around $(S^1)^N$, i.e. each component $\theta_i  = \tilde{\theta}_i \mod 2\pi$ for $i=1,\dots,N$, we obtain the following equation 
\begin{align} \label{PDEqN}
\partial_t q^N - \sum_{i=1}^N D \partial_{\theta_i}^2 q^N + \sum_{i=1}^N \partial_{\theta_i} \left( q^N \left(u_1 (t,\theta_i) +\frac{u_2 (t,\theta_i)}{N} \sum_{j=1}^N \sin (\theta_j - \theta_i - \alpha) \right) \right) = 0,
\end{align}
where $q^N$ denotes the wrapped joint distribution of $\theta_i, \ i = 1, \dots, N$ with initial data $q^N \vert_{t=0} = q_0^{\otimes N} = q_0(\theta_1) \dots q_0 (\theta_N)$. We denote the first marginal of $q^{N} [u_1,u_2]$ by 
\begin{align*}
q^{N;1} [u_1,u_2] (t,\theta) = \int_{(S^1)^{N-1}} q^{N} [u_1,u_2] (t, \theta, \theta_2, \dots, \theta_N) \, d\theta_2 \dots d\theta_N
\end{align*}
and consider the cost functional on the microscopic level (parallel to the cost functional (\ref{costfunc}) for the mean-field PDE (\ref{mfpde})):
\begin{align} \label{JNfunctional}
\begin{split}
\mathcal{J}_N (q^N[&u_1,u_2], u_1, u_2) = \frac{\alpha_r}{2} \int_0^T \int_{S^1}  (q^{N;1}(t,\theta) - z(t,\theta))^2 \, d\theta dt \\
&+ \frac{\alpha_t}{2} \int_{S^1} (q^{N;1}(T,\theta) - z(T,\theta))^2 \, d\theta  + \frac{1}{2} \int_0^T \int_{S^1} (\beta_1 u_1^2 (t,\theta) + \beta_2 u^2_2(t,\theta)) \, d\theta dt \\
& \quad \quad \quad \quad \quad \quad =:  \mathcal{J}^r_N + \mathcal{J}^t_N + \mathcal{J}_N^u. 
\end{split}
\end{align}

The minimization of the cost functional describes that the joint distribution must achieve the given distribution $z$, where we give particular importance to the time $T$, with minimal cost associated to the control functions $u_1,u_2$. For simplicity, we assume $z \in L^\infty(0,T; L^\infty (S^1))$. Furthermore, we define the space of control functions as
\begin{align} \label{cfspace}
\mathcal{U} = \left\lbrace u \in L^\infty (0,T; W^{1,q} (S^1)) \ | \ ||u||_{L^\infty (0,T; W^{1,q} (S^1))} \leq M \right\rbrace
\end{align} 
for $q>2$ and some uniform constant $M >0$. It follows directly from Sobolev imbedding theorem for compact manifolds (see \cite{MR1636569}) that for every function $u\in \mathcal{U}$ it holds
\begin{align} \label{ubound}
||u||_{L^\infty (0,T; L^\infty (S^1))}< C_{\text{Sobolev}}M =: \tilde{M}.
\end{align}
This choice of the space of control functions is the least restrictive as it only takes into account the necessary conditions to allow the application of our strategy. \\

The main result of this paper is the following theorem.

\begin{theorem} \label{mainresult}
Assume that the initial data is probability density $q_0 \in H^1(S^1)$. Then for any fixed $N$, the cost functional $\mathcal{J}_N (q^N[u_1,u_2], u_1, u_2)$ has a pair of minimizers $u^N_1,u^N_2 \in \mathcal{U}$, i.e.
\begin{align*}
\mathcal{J}_N (q^N[u^N_1,u^N_2], u^N_1, u^N_2) = \min_{u_1,u_2 \in \mathcal{U}}\mathcal{J}_N (q^N[u_1,u_2], u_1, u_2),
\end{align*}
where $q^N[u_1,u_2]$ is the solution of (\ref{PDEqN}) for some control functions $u_1,u_2 \in \mathcal{U}$. Moreover, any weak accumulation point $(\overline{u}_1,\overline{u}_2) \in \mathcal{U}^2$ of $(u^N_1,u^N_2)_{N \in \mathbb{N}}$ is a minimizer of the cost functional $\mathcal{J}$ and it holds
\begin{align*}
\mathcal{J} (q [\overline{u}_1,\overline{u}_2], \overline{u}_1, \overline{u}_2) = \min_{u_1,u_2 \in \mathcal{U}}\mathcal{J} (q [u_1,u_2], u_1, u_2).
\end{align*}
\end{theorem}

The proof idea of the main result and the arrangement of the paper is as follows. In \textbf{Section 2} we prove solvability of the microscopic control problem. This gives us a sequence $(u_1^N, u_2^N)_{N \in \mathbb{N}}$, where each $u_1^N, u_2^N$ minimizes the corresponding cost functional $\mathcal{J}_N$. We aim to prove that any weak accumulation point $(\overline{u}_1,\overline{u}_2)$ of this sequence minimizes the cost functional $\mathcal{J}$. This proof is based on the idea of $\Gamma$-convergence which was first introduced by Ennio De Giorgi. We refer to \cite{MR2836339} for a source in English. In the first step we show inequality
\begin{align*}
\liminf_{k \rightarrow \infty} \mathcal{J}_{N_k} (q^{N_k} [u_1^{N_k},u_2^{N_k}],u_1^{N_k},u_2^{N_k}) \geq \mathcal{J} (q[\overline{u}_1,\overline{u}_2],\overline{u}_1,\overline{u}_2),
\end{align*}
which means that $\mathcal{J}$ provides an asymptotic lower bound for $\mathcal{J}_{N_k}$. A crucial tool for this step is the strong convergence of $q^{N;1} \rightarrow q$ in $L^\infty (0,T; L^1(S^1))$ for fixed control functions, which we prove in \textbf{Section 3}. Hereby we make use of the large deviation estimate given in \cite{MR3858403}. In the second step we show that $(\overline{u}_1,\overline{u}_2)$ minimizes the cost functional $\mathcal{J}$. This step is based on the following asymptotic result: For any sequence $(\hat{u}^N_1,\hat{u}^N_2)_{N \in \mathbb{N}}$ in $\mathcal{U}^2$ it holds
\begin{align*}
\lim_{N \rightarrow \infty} \left( \mathcal{J}_N (q^N [\hat{u}^N_1,\hat{u}^N_2],\hat{u}^N_1,\hat{u}^N_2) - \mathcal{J} (q [\hat{u}^N_1, \hat{u}^N_2], \hat{u}^N_1, \hat{u}^N_2) \right) = 0.
\end{align*}
We provide the details of this intermediate result and the main result in \textbf{Section 4}.  \\

As far as we are aware, no rigorous analysis of the mean-field optimal control problem for Kuramoto oscillators has been done so far. Although the interaction force is given by the sine function, which is bounded Lipschitz continuous, the fact that the Kuramoto oscillators naturally have range on the periodic domain $S^1$ does not allow direct application of existing results for mean-field optimal control on $\mathbb{R}$, e.g. \cite{chen2024meanfieldcontroldiffusionaggregation}.  
To overcome this issue, instead of considering the control problem for particle trajectories, we work directly with the corresponding Liouville equation. An established technique in mean-field control studies is using weak convergence of empirical measure, which relies on particle trajectories. Since we are working with the Liouville equation instead, we need to obtain the respective propagation of chaos result. For this reason, the strong $L^1$ convergence of the first marginal plays an important role in the discussion.

\section{Optimal Control Problem for the Liouville Equation}

In this section, we study the optimal control problem on the particle level. The well-posedness of the Liouville equation  (\ref{PDEqN}) for fixed $N \in \mathbb{N}$ can be obtained directly by classical results for linear second order parabolic equations, for example Theorem 4.2 in Chapter III of \cite{MR241822}. We denote its weak solution by $q[u_1,u_2]$. The first marginal $q^{N;1}[u_1,u_2]$ satisfies the following equation, by integrating equation (\ref{PDEqN}) with respect to $\theta_2, \dots, \theta_N$:
\begin{align} \label{marginalequation}
\begin{split}
&\partial_t q^{N;1}[u_1,u_2] - D \partial^2_\theta q^{N;1}[u_1,u_2] + \partial_\theta \left(  q^{N;1}[u_1,u_2] u_1 \right) \\
&+ \partial_\theta \left(\frac{N-1}{N} u_2 \int_{S^1} q^{N;2}[u_1,u_2](\theta, \tilde{\theta}) \sin (\tilde{\theta} - \theta - \alpha) \, d\tilde{\theta} + \frac{1}{N} u_2  \sin (-\alpha) q^{N;1}[u_1,u_2]  \right) = 0.
\end{split}
\end{align}

It is easy to notice that the equations (\ref{PDEqN}) and (\ref{marginalequation}) both conserve the total mass, and the corresponding solutions are therefore density functions as long as the initial data is a density function.  \\

Analogously, $q^{N;2}[u_1,u_2]$ represents the second marginal density of $q^N [u_1,u_2]$, i.e.
\begin{align*}
q^{N;2}[u_1,u_2] (t,\theta, \tilde{\theta}) = \int_{(S^1)^{N-2}} q^{N} [u_1,u_2] (t,\theta,\tilde{\theta},\theta_3, \dots, \theta_N) \, d\theta_3 \dots d\theta_N.
\end{align*}

The main result of this section is
\begin{proposition} \label{JNminimizer}
Assume that the initial data is probability density $q_0 \in H^1(S^1)$. The functional $\mathcal{J}_N$ admits a minimizer in \(\mathcal{U}^2\), i.e. there exists a pair $(u_1^*, u_2^*) \in \mathcal{U}^2$ such that
$$ \mathcal{J}_N (q^N[u_1^*, u_2^*], u_1^*, u_2^*) = \inf_{u_1, u_2 \in \mathcal{U}} \mathcal{J}_N (q^N[u_1, u_2], u_1, u_2).$$
\end{proposition}

\begin{proof}
Let $(u_1^j, u_2^j)_{j \in \mathbb{N}} \in \mathcal{U}^2$ be a minimizing sequence of the given optimization problem, namely
\begin{align*}
\lim_{j \rightarrow \infty} \mathcal{J}_N (q^N [u_1^j, u_2^j],u_1^j, u_2^j) = \inf_{u_1,u_2 \in \mathcal{U}} \mathcal{J}_N (q^N [u_1, u_2],u_1, u_2).
\end{align*}
The sequence $(u_1^j, u_2^j)_{j \in \mathbb{N}}$ is bounded in $(L^\infty (0,T; W^{1,q}(S^1)))^2$. Hence there is a weakly* convergent (not relabeled) subsequence with weak limit $(u_1^*, u_2^*) \in \mathcal{U}^2$. We aim to show that $(u_1^*, u_2^*)$ minimizes the cost functional (\ref{JNfunctional}), i.e.
\begin{align} \label{ineq3}
\lim_{j \rightarrow \infty} \mathcal{J}_N (q^N [u_1^j, u_2^j], u_1^j, u_2^j) \geq \mathcal{J}_N (q^N [u_1^*, u_2^*], u_1^*, u_2^*).
\end{align}
Given its structure, we exploit lower semicontinuity of the $L^2(0,T; L^2(S^1))$ norm to deal with the term $\mathcal{J}_N^u$. The terms $\mathcal{J}_N^r$ and $\mathcal{J}_N^t$ require strong convergence of $\left(q^{N;1} [u_1^j, u_2^j] \right)_{j \in \mathbb{N}}$ in $C(0,T;L^2(S^1))$. Continuity with respect to the time parameter is crucial, since the term $\mathcal{J}_N^t$ in the cost functional pays special attention to the final time $T$. Moreover, in our case strong convergence is necessary to take limits in terms $\mathcal{J}_N^r$ and $\mathcal{J}_N^t$, whereas weak convergence would be sufficient without the presence of control functions in the model. We use Aubins-Lions Lemma to obtain strong convergence of a subsequence in $C(0,T;L^2(S^1))$. In order to do that, we need to show that the sequences $\left( q^{N;1}[u_1^j,u_2^j] \right)_{j \in \mathbb{N}}$ and $\left( \partial_\theta q^{N;1}[u_1^j,u_2^j] \right)_{j \in \mathbb{N}}$ are both bounded in $L^\infty (0,T; L^2(S^1))$. The first one is fairly straightforward, whereas the second one requires additional steps and techniques because in this case the derivative is applied to the nonlocal term. An important step is deriving equations corresponding to $q^{N;2} [u_1^j, u_2^j]$ and showing that the sequence $\left( \partial_{\theta} q^{N;2} [u_1^j, u_2^j] \right)_{j \in \mathbb{N}}$ is bounded in $L^2(0,T; H^1((S^1)^2))$. \\

We start by noticing that by integrating equation (\ref{PDEqN}) with respect to $\theta_3, \dots, \theta_N$, $q^{N;2} [u_1^j, u_2^j]$ satisfies
\begin{align} \label{secondmarginal}
\begin{split}
&\partial_t q^{N;2}[u_1,u_2](\theta, \tilde{\theta}) - D \left( \partial^2_\theta q^{N;2} [u_1,u_2](\theta,\tilde{\theta}) +\partial^2_{\tilde{\theta}} q^{N;2}[u_1,u_2] (\theta,\tilde{\theta}) \right) \\
&+ \partial_\theta \left(  q^{N;2}[u_1,u_2](\theta,\tilde{\theta}) u_1 (\theta) \right)  + \partial_{\tilde{\theta}} \left(  q^{N;2}[u_1,u_2](\theta,\tilde{\theta}) u_1 (\tilde{\theta}) \right)\\
&+ \partial_\theta \left(\frac{N-2}{N} u_2 (\theta) \int_{S^1} q^{N;3}[u_1,u_2] (\theta, \tilde{\theta},\hat{\theta}) \sin (\hat{\theta} - \theta - \alpha) \, d\hat{\theta} \right)  \\
&+ \frac{1}{N} \partial_\theta \left( u_2 (\theta) \sin (-\alpha) q^{N;2} [u_1,u_2](\theta, \tilde{\theta}) + u_2 (\theta) \sin (\tilde{\theta} - \theta -\alpha) q^{N;2} [u_1,u_2](\theta, \tilde{\theta}) \right)   \\
&+ \partial_{\tilde{\theta}} \left(\frac{N-2}{N} u_2 (\tilde{\theta}) \int_{S^1} q^{N;3}[u_1,u_2] (\theta, \tilde{\theta},\hat{\theta}) \sin (\hat{\theta} - \tilde{\theta} - \alpha) \, d\hat{\theta} \right) \\
&+ \frac{1}{N} \partial_{\tilde{\theta}} \left( u_2 (\tilde{\theta}) \sin (-\alpha) q^{N;2} [u_1,u_2](\theta, \tilde{\theta}) + u_2 (\tilde{\theta}) \sin (\theta - \tilde{\theta} -\alpha) q^{N;2}[u_1,u_2] (\theta, \tilde{\theta}) \right) = 0,
\end{split}
\end{align}
where $q^{N;3}[u_1,u_2]$ represents the third marginal of $q^N[u_1,u_2]$, i.e.
\begin{align*}
q^{N;3}[u_1,u_2] (t,\theta, \tilde{\theta}, \hat{\theta}) = \int_{(S^1)^{N-3}} q^{N} [u_1,u_2] (t,\theta,\tilde{\theta}, \hat{\theta}, \theta_4, \dots, \theta_N) \, d\theta_4 \dots d\theta_N.
\end{align*}

We multiply the equation (\ref{secondmarginal}) by $q^{N;2} [u_1^j, u_2^j]$ and integrate by parts to obtain
\begin{align*}
\begin{split}
&\frac{1}{2} \frac{d}{dt} \int_{(S^1)^2} (q^{N;2}[u_1^j, u_2^j](\theta, \tilde{\theta}))^2 \, d\theta d\tilde{\theta} +  D \int_{(S^1)^2} ( \partial_\theta q^{N;2}[u_1^j, u_2^j] (\theta, \tilde{\theta}))^2 + ( \partial_{\tilde{\theta}} q^{N;2}[u_1^j, u_2^j](\theta, \tilde{\theta}))^2  \, d\theta d\tilde{\theta} \\
= & \int_{(S^1)^2} \left( q^{N;2}[u_1^j, u_2^j](\theta, \tilde{\theta}) u^j_1(\theta) \partial_\theta q^{N;2}[u_1^j, u_2^j] (\theta, \tilde{\theta}) + q^{N;2}[u_1^j, u_2^j](\theta, \tilde{\theta}) u^j_1 (\tilde{\theta}) \partial_{\tilde{\theta}} q^{N;2}[u_1^j, u_2^j] (\theta, \tilde{\theta}) \right) \, d\theta d\tilde{\theta} \\
+ & \int_{(S^1)^2} \frac{N-2}{N} u_2^j (\theta) \left( \int_{S^1} q^{N;3}[u_1^j, u_2^j](\theta, \tilde{\theta}, \hat{\theta}) \sin (\hat{\theta} - \theta - \alpha) \, d\hat{\theta} \right)  \partial_\theta q^{N;2}[u_1^j, u_2^j] (\theta,\tilde{\theta})   \,  d\theta d\tilde{\theta} \\
+ & \int_{(S^1)^2} \frac{N-2}{N} u_2^j (\tilde{\theta}) \left( \int_{S^1} q^{N;3}[u_1^j, u_2^j](\theta, \tilde{\theta}, \hat{\theta}) \sin (\hat{\theta} - \tilde{\theta} - \alpha) \, d\hat{\theta} \right) \partial_{\tilde{\theta}} q^{N;2}[u_1^j, u_2^j] (\theta,\tilde{\theta})   \,  d\theta d\tilde{\theta} \\
+ &\frac{1}{N} \int_{(S^1)^2} \left( u_2^j (\theta) \sin (-\alpha) q^{N;2}[u_1^j, u_2^j] (\theta,\tilde{\theta}) + u_2^j (\theta) \sin (\tilde{\theta} - \theta -\alpha) q^{N;2}[u_1^j, u_2^j] (\theta,\tilde{\theta}) \right) \partial_\theta q^{N;2}[u_1^j, u_2^j] \,    d\theta d\tilde{\theta} \\
+ &\frac{1}{N} \int_{(S^1)^2} \left( u_2^j (\tilde{\theta}) \sin (-\alpha) q^{N;2}[u_1^j, u_2^j] (\theta,\tilde{\theta}) + u_2^j (\tilde{\theta}) \sin (\theta - \tilde{\theta} -\alpha) q^{N;2}[u_1^j, u_2^j] (\theta,\tilde{\theta}) \right) \partial_{\tilde{\theta}} q^{N;2}[u_1^j, u_2^j] \,    d\theta d\tilde{\theta}.
\end{split}
\end{align*}

We integrate this equation on $[0,t]$ for every $t\in (0,T]$ and use (\ref{ubound}) to obtain
\begin{align*}
\begin{split}
&\frac{1}{2} \int_{(S^1)^2} (q^{N;2}[u_1^j, u_2^j](t))^2 \, d\theta d\tilde{\theta} - \frac{1}{2} \int_{(S^1)^2} (q^{N;2}[u_1^j, u_2^j](0))^2 \, d\theta d\tilde{\theta} \\
&+  D \int_0^t \int_{(S^1)^2} ( \partial_\theta q^{N;2}[u_1^j, u_2^j])^2  + ( \partial_{\tilde{\theta}} q^{N;2}[u_1^j, u_2^j])^2 \, d\theta d\tilde{\theta} ds \\
\leq &  ||q^{N;2}[u_1^j, u_2^j]||_{L^2(0,t;L^2((S^1)^2))} ||u_1^j||_{L^\infty(0,t;L^\infty(S^1))}  ||\partial_\theta q^{N;2}[u_1^j, u_2^j]||_{L^2(0,t;L^2((S^1)^2))} \\
&+ ||q^{N;2}[u_1^j, u_2^j]||_{L^2(0,t;L^2((S^1)^2))} ||u_1^j||_{L^\infty(0,t;L^\infty(S^1))}  ||\partial_{\tilde{\theta}} q^{N;2}[u_1^j, u_2^j]||_{L^2(0,t;L^2((S^1)^2))} \\
&+ ||q^{N;2}[u_1^j, u_2^j]||_{L^2(0,t;L^2((S^1)^2))} ||u_2^j||_{L^\infty(0,t;L^\infty(S^1))}  ||\partial_\theta q^{N;2}[u_1^j, u_2^j]||_{L^2(0,t;L^2((S^1)^2))} \\
&+  ||q^{N;2}[u_1^j, u_2^j]||_{L^2(0,t;L^2((S^1)^2))} ||u_2^j||_{L^\infty(0,t;L^\infty(S^1))}  ||\partial_{\tilde{\theta}} q^{N;2}[u_1^j, u_2^j]||_{L^2(0,t;L^2((S^1)^2))} \\
\leq & 2\tilde{M} ||q^{N;2}[u_1^j, u_2^j]||_{L^2(0,t;L^2((S^1)^2))} \left( ||\partial_\theta q^{N;2}[u_1^j, u_2^j]||_{L^2(0,t;L^2((S^1)^2))}  + ||\partial_{\tilde{\theta}} q^{N;2}[u_1^j, u_2^j]||_{L^2(0,t;L^2((S^1)^2))} \right) \\
\leq & \frac{\tilde{M}}{D} ||q^{N;2}[u_1^j, u_2^j]||_{L^2(0,t;L^2((S^1)^2))}^2 \\
&+ \frac{D}{2} \left( ||\partial_\theta q^{N;2}[u_1^j, u_2^j]||_{L^2(0,t;L^2((S^1)^2))}^2 +  ||\partial_{\tilde{\theta}} q^{N;2}[u_1^j, u_2^j]||^2_{L^2(0,t;L^2((S^1)^2))}  \right) ,
\end{split}
\end{align*}
where we used Young's inequality in the last step. Now we can absorb the second term on the right-hand side and use Grönwall's inequality to obtain
\begin{align*}
||q^{N;2}[u_1^j,u_2^j](t)||_{L^2((S^1)^2)}^2 \leq C, \ \forall t \in (0,T],
\end{align*}
where the constant $C>0$ does not depend on $j$. Moreover, we see directly that it holds
\begin{align} \label{secmarineq}
||\partial_\theta q^{N;2}[u_1^j,u_2^j]||_{L^2(0,T;L^2((S^1)^2))}^2 +  ||\partial_{\tilde{\theta}} q^{N;2}[u_1^j, u_2^j]||^2_{L^2(0,T;L^2((S^1)^2))} \leq C
\end{align}
with a uniform constant $C>0$. Hence, the sequence $\left( q^{N;2} [u_1^j, u_2^j] \right)_{j \in \mathbb{N}}$ is bounded in $L^2(0,T; H^1((S^1)^2))$. \\

Now we multiply the equation (\ref{marginalequation}) by $q^{N;1} [u_1^j, u_2^j]$, integrate by parts and use (\ref{ubound}) to obtain
\begin{align} \label{particle1}
\begin{split}
&\frac{1}{2} \frac{d}{dt} \int_{S^1} (q^{N;1}[u_1^j, u_2^j])^2 \, d\theta  + D \int_{S^1} (\partial_\theta q^{N;1}[u_1^j, u_2^j])^2 \, d\theta  \\
= & \int_{S^1} \left( q^{N;1} [u_1^j, u_2^j]  u_1^j +  \frac{N-1}{N}u_2^j  \int_{S^1} q^{N;2} [u_1^j, u_2^j]  \sin (\tilde{\theta} - \theta - \alpha) \, d\tilde{\theta} \right) \partial_\theta q^{N;1} [u_1^j, u_2^j] \, d\theta \\
&+ \frac{1}{N} \int_{S^1} u_2^j \sin (-\alpha) q^{N;1} [u_1^j,u_2^j] \partial_\theta q^{N;1} [u_1^j, u_2^j] \, d\theta \\
\leq & 2\tilde{M} ||q^{N;1} [u_1^j, u_2^j](t)||_{L^2(S^1)} ||\partial_\theta q^{N;1} [u_1^j, u_2^j] (t)||_{L^2(S^1)} \\
\leq & \frac{\tilde{M}}{D} ||q^{N;1} [u_1^j, u_2^j](t)||_{L^2(S^1)}^2 + \frac{D}{2} ||\partial_\theta q^{N;1} [u_1^j, u_2^j] (t)||^2_{L^2(S^1)},
\end{split}
\end{align}
where we used Young's inequality again. After absorbing the second term on the right-hand side, we use Grönwall's inequality to obtain
\begin{align} \label{bound1}
|| q^{N;1} [u_1^j, u_2^j](t)||^2_{L^2(S^1)} \leq C, \ \forall t \in (0,T],
\end{align}
where the constant $C$ does not depend on $j$.\\ 

Now we differentiate the equation (\ref{marginalequation}) with respect to $\theta$, multiply it by $\partial_\theta q^{N;1}[u_1^j, u_2^j]$ and integrate by parts to obtain
\begin{align*}
&\frac{1}{2} \frac{d}{dt} \int_{S^1} (\partial_\theta q^{N;1}[u_1^j, u_2^j])^2 \, d\theta  + D \int_{S^1} (\partial_\theta^2 q^{N;1}[u_1^j, u_2^j])^2 \, d\theta  \\
= & \int_{S^1} \partial_\theta q^{N;1} [u_1^j, u_2^j] u_1^j \partial^2_\theta q^{N;1} [u_1^j, u_2^j] \, d\theta + \int_{S^1} q^{N;1} [u_1^j, u_2^j] \partial_\theta u_1^j  \partial^2_\theta q^{N;1} [u_1^j, u_2^j] \, d\theta  \\
&+ \int_{S^1} \frac{N-1}{N} \partial_\theta u_2^j \left( \int_{S^1} q^{N;2} [u_1^j, u_2^j]  \sin (\tilde{\theta} - \theta - \alpha) \, d\tilde{\theta} \right) \partial^2_\theta q^{N;1} [u_1^j, u_2^j] \, d\theta  \\
&+ \int_{S^1} \frac{N-1}{N} u_2^j \left( \int_{S^1}   \partial_\theta q^{N;2} [u_1^j, u_2^j] \sin (\tilde{\theta} - \theta - \alpha) \, d\tilde{\theta} \right) \partial^2_\theta q^{N;1} [u_1^j, u_2^j] \, d\theta \\
&- \int_{S^1} \frac{N-1}{N} u_2^j \left( \int_{S^1}    q^{N;2} [u_1^j, u_2^j] \cos (\tilde{\theta} - \theta - \alpha) \, d\tilde{\theta} \right) \partial^2_\theta q^{N;1} [u_1^j, u_2^j] \, d\theta \\
&+ \frac{1}{N} \int_{S^1} \partial_\theta u_2^j \sin (-\alpha) q^{N;1} [u_1^j, u_2^j]   \partial^2_\theta q^{N;1} [u_1^j, u_2^j]  \, d\theta + \frac{1}{N} \int_{S^1} u_2^j  \sin (-\alpha) \partial_\theta q^{N;1} [u_1^j, u_2^j]   \partial^2_\theta q^{N;1} [u_1^j, u_2^j] \, d\theta.
\end{align*}

Integrating on $[0,t]$ for $t \in (0,T]$, we obtain
\begin{align*}
&\frac{1}{2} ||\partial_\theta q^{N;1}[u_1^j, u_2^j](t)||^2_{L^2(S^1)} - \frac{1}{2} ||\partial_\theta q^{N;1}[u_1^j, u_2^j](0)||^2_{L^2(S^1)} + D ||\partial_\theta^2 q^{N;1}[u_1^j, u_2^j]||^2_{L^2(0,T;L^2(S^1))} \\
\leq & || \partial_\theta q^{N;1} [u_1^j, u_2^j]||_{L^2(0,t; L^2(S^1))} ||u_1^j||_{L^\infty(0,t; L^\infty(S^1))} ||\partial^2_\theta q^{N;1} [u_1^j, u_2^j] ||_{L^2(0,t; L^2(S^1))} \\
&+  ||q^{N;1} [u_1^j, u_2^j]||_{L^2(0,t; L^{\frac{2q}{q-2}}(S^1))}  ||\partial_\theta u_1^j||_{L^\infty(0,t; L^q(S^1))} ||\partial^2_\theta q^{N;1} [u_1^j, u_2^j] ||_{L^2(0,t; L^2(S^1))}  \\
&+ \frac{N-1}{N} ||\partial_\theta u_2^j ||_{L^\infty(0,t; L^q(S^1))} || q^{N;1} [u_1^j, u_2^j]||_{L^2(0,t; L^{\frac{2q}{q-2}}(S^1))} ||  \partial^2_\theta q^{N;1} [u_1^j, u_2^j]||_{L^2(0,t; L^2(S^1))}   \\
&+ \frac{N-1}{N} ||u_2^j ||_{L^\infty(0,t; L^\infty(S^1))}  \left\vert \left\vert \int_{S^1}   \partial_\theta q^{N;2} [u_1^j, u_2^j] \sin (\tilde{\theta} - \theta - \alpha) \, d\tilde{\theta} \right\vert \right\vert_{L^2(0,t; L^2(S^1))} \\
  &\cdot ||\partial^2_\theta q^{N;1} [u_1^j, u_2^j] ||_{L^2(0,t; L^2(S^1))}  \\
&+ \frac{N-1}{N} ||u_2^j ||_{L^\infty(0,t; L^\infty(S^1))}  ||q^{N;1} [u_1^j, u_2^j]||_{L^2(0,t; L^2(S^1))}   ||\partial^2_\theta q^{N;1} [u_1^j, u_2^j]|| _{L^2(0,t; L^2(S^1))}  \\
&+ \frac{1}{N} || \partial_\theta u_2^j ||_{L^\infty(0,t; L^q(S^1))}  ||q^{N;1}[u_1^j, u_2^j]||_{L^2(0,t; L^{\frac{2q}{q-2}}(S^1))}  ||\partial^2_\theta q^{N;1} [u_1^j, u_2^j]||_{L^2(0,t; L^2(S^1))}  \\
&+ \frac{1}{N} ||u_2^j ||_{L^\infty(0,t; L^\infty(S^1))}  ||\partial_\theta  q^{N;1}[u_1^j, u_2^j] ||_{L^2(0,t; L^2(S^1))}  ||\partial^2_\theta q^{N;1} [u_1^j, u_2^j] ||_{L^2(0,t; L^2(S^1))}.
\end{align*}

Using Gagliardo-Nirenberg inequality, we observe
\begin{align*}
||q^{N;1} [u_1^j, u_2^j]||^2_{L^2(0,t; L^{\frac{2q}{q-2}}(S^1))} \leq & \int_0^t C ||\partial_\theta q^{N;1} [u_1^j, u_2^j](s) ||_{ L^2(S^1)}^{\frac{2}{q}} ||q^{N;1} [u_1^j, u_2^j](s) ||_{ L^2(S^1)}^{\frac{2q-2}{q}} \, ds \\
\leq & \int_0^t C ||\partial_\theta q^{N;1} [u_1^j, u_2^j](s) ||_{ L^2(S^1)}^{\frac{2}{q}}  \, ds \\
\leq &\int_0^t C (1 + ||\partial_\theta q^{N;1} [u_1^j, u_2^j](s) ||_{ L^2(S^1)}^2)^\frac{1}{q}  \, ds \\
\leq & \int_0^t C (1 + ||\partial_\theta q^{N;1} [u_1^j, u_2^j](s) ||_{ L^2(S^1)}^2)  \, ds \\
=& CT + C ||\partial_\theta q^{N;1} [u_1^j, u_2^j]||^2_{L^2(0,t;L^2(S^1))} \\
\leq &C\left( 1+  ||\partial_\theta q^{N;1} [u_1^j, u_2^j](s) ||_{ L^2(S^1)}^2\right)
\end{align*}
for some constant $C>0$, since we know from (\ref{bound1}) that $||q^{N;1} [u_1^j, u_2^j] ||_{L^\infty(0,t; L^2(S^1))}$ are uniformly bounded. Moreover, we can estimate
\begin{align*}
&\left\vert \left\vert \int_{S^1}   \partial_\theta q^{N;2} [u_1^j, u_2^j] (t,\theta, \tilde{\theta}) \sin (\tilde{\theta} - \theta - \alpha) \, d\tilde{\theta} \right\vert \right\vert_{L^2(0,T; L^2(S^1))}^2 \\
\leq & \int_0^t \int_{S^1} \left( \int_{S^1}   (\partial_\theta q^{N;2} [u_1^j, u_2^j])^2 \, d\tilde{\theta} \right) \left( \int_{S^1} ( \sin (\tilde{\theta} - \theta - \alpha))^2 \, d\tilde{\theta} \right) \, d\theta ds \\
\leq &  \int_0^t \int_{(S^1)^2} 2\pi (\partial_\theta q^{N;2} [u_1^j, u_2^j])^2 \, d\tilde{\theta} d\theta ds \leq  C,
\end{align*}
since we have already established (\ref{secmarineq}). Therefore we can conclude
\begin{align*}
&\frac{1}{2} ||\partial_\theta q^{N;1}[u_1^j, u_2^j](t)||^2_{L^2(S^1)} - \frac{1}{2} ||\partial_\theta q^{N;1}[u_1^j, u_2^j](0)||^2_{L^2(S^1)} + D ||\partial_\theta^2 q^{N;1}[u_1^j, u_2^j]||_{L^2(0,t;L^2(S^1))}^2 \\
\leq & \tilde{M} || \partial_\theta q^{N;1} [u_1^j, u_2^j]||_{L^2(0,t; L^2(S^1))} ||\partial^2_\theta q^{N;1} [u_1^j, u_2^j]||_{L^2(0,t; L^2(S^1))} \\
+ &  M || q^{N;1} [u_1^j, u_2^j]||_{L^2(0,t; L^{\frac{2q}{q-2}} (S^1)} ||\partial^2_\theta q^{N;1} [u_1^j, u_2^j] ||_{L^2(0,t; L^2(S^1))}  \\
+ & \frac{N-1}{N}  M  ||q^{N;1} [u_1^j, u_2^j]||_{L^2(0,t; L^{\frac{2q}{q-2}}(S^1))} ||  \partial^2_\theta q^{N;1} [u_1^j, u_2^j]||_{L^2(0,t; L^2(S^1))}   \\
+ & \frac{N-1}{N} C \tilde{M}  ||\partial^2_\theta q^{N;1} [u_1^j, u_2^j] ||_{L^2(0,t; L^2(S^1))} +  \frac{N-1}{N} C\tilde{M} ||\partial^2_\theta q^{N;1} [u_1^j, u_2^j]||_{L^2(0,t; L^2(S^1))}  \\
+& \frac{1}{N} M ||q^{N;1}[u_1^j, u_2^j]||_{L^2(0,t; L^{\frac{2q}{q-2}}(S^1))}  ||\partial^2_\theta q^{N;1} [u_1^j, u_2^j]||_{L^2(0,t; L^2(S^1))}  \\
+& \frac{1}{N} \tilde{M}  ||\partial_\theta  q^{N;1}[u_1^j, u_2^j] ||_{L^2(0,t; L^2(S^1))}  ||\partial^2_\theta q^{N;1} [u_1^j, u_2^j] ||_{L^2(0,t; L^2(S^1))} 
\end{align*}

Therefore
\begin{align*}
&\frac{1}{2} ||\partial_\theta q^{N;1}[u_1^j, u_2^j](t)||^2_{L^2(S^1)} - \frac{1}{2} ||\partial_\theta q^{N;1}[u_1^j, u_2^j](0)||^2_{L^2(S^1)} + D ||\partial_\theta^2 q^{N;1}[u_1^j, u_2^j]||_{L^2(0,t;L^2(S^1))}^2 \\
\leq & \tilde{M} \left( C_\epsilon || \partial_\theta q^{N;1} [u_1^j, u_2^j]||_{L^2(0,t; L^2(S^1))}^2 + \epsilon ||\partial^2_\theta q^{N;1} [u_1^j, u_2^j]||_{L^2(0,t; L^2(S^1))}^2 \right) \\
+ & M \left( C_\epsilon \left(1+ ||\partial_\theta q^{N;1} [u_1^j, u_2^j]||_{L^2(0,t; L^2 (S^1)}^2 \right) + \epsilon||\partial^2_\theta q^{N;1} [u_1^j, u_2^j] ||_{L^2(0,t; L^2(S^1))}^2 \right)  \\
+ & \frac{N-1}{N} M \left( C_\epsilon \left( 1+ ||\partial_\theta q^{N;1} [u_1^j, u_2^j]||^2_{L^2(0,t;L^2 (S^1))} \right) + \epsilon ||  \partial^2_\theta q^{N;1} [u_1^j, u_2^j]||_{L^2(0,t; L^2(S^1))}^2 \right)   \\
+ & \frac{N-1}{N} \sqrt{2C\pi}\tilde{M} \left( C_\epsilon + \epsilon  ||\partial^2_\theta q^{N;1} [u_1^j, u_2^j] ||_{L^2(0,t; L^2(S^1))}^2 \right) \\
 +  &\frac{N-1}{N} C \tilde{M} \left( C_\epsilon + \epsilon||\partial^2_\theta q^{N;1} [u_1^j, u_2^j] ||_{L^2(0,t; L^2(S^1))}^2 \right)  \\
+& \frac{1}{N} M \left( C_\epsilon \left( 1+ ||\partial_\theta q^{N;1}||_{L^2(0,t; L^2(S^1))}^2 \right) +  \epsilon||\partial^2_\theta q^{N;1} [u_1^j, u_2^j] ||_{L^2(0,t; L^2(S^1))}^2 \right)  \\
+& \frac{1}{N} \tilde{M} \left( C_\epsilon ||\partial_\theta  q^{N;1} ||_{L^2(0,t; L^2(S^1))}^2 + \epsilon ||\partial^2_\theta q^{N;1} [u_1^j, u_2^j] ||_{L^2(0,t; L^2(S^1))}^2 \right)  \\
\leq & C + C ||\partial_\theta q^{N;1}||_{L^2(0,t; L^2(S^1))}^2 \\
+ &\epsilon \left( 2M + \frac{N+1}{N} \tilde{M} + \frac{N-1}{N} C\tilde{M} + \frac{N-1}{N} \sqrt{2C\pi} \tilde{M} \right) ||\partial^2_\theta q^{N;1} [u_1^j, u_2^j] ||_{L^2(0,t; L^2(S^1))}^2
\end{align*}
for every $\epsilon>0$ and for some constant $C_\epsilon >0$ and some constant $C>0$ that does not depend on $j$ or $t$. Choosing $\epsilon$ as
\begin{align*}
\epsilon = \frac{D}{2}  \left( 2M + 2 \tilde{M} + C \tilde{M} +  \sqrt{2C\pi} \tilde{M} \right)^{-1},
\end{align*}
we can absorb the term $||\partial_\theta^2 q^{N;1}[u_1^j, u_2^j]||_{L^2(0,t;L^2(S^1))}^2$ and after applying (\ref{secmarineq}) and Grönwall's inequality, we obtain
\begin{align*}
||\partial_\theta q^{N;1}[u_1^j, u_2^j](t)||^2_{L^2(S^1)} \leq C, \ \forall t \in (0,T]
\end{align*}
uniformly in $j$. Therefore the sequence $(q^{N;1}[u_1^j, u_2^j])_{j\in \mathbb{N}}$ is uniformly bounded in $L^\infty (0,T; H^1(S^1))$. Using this fact, the PDE (\ref{marginalequation}) and the space of control functions (\ref{cfspace}), we get immediately that the sequence $(\partial_t q^{N;1} [u_1^j, u_2^j])_{j\in \mathbb{N}}$ is bounded in $L^2 (0,T; H^{-1}(S^1))$. Then we use the Aubin-Lions Lemma and obtain a strongly convergent (not relabeled) subsequence
\begin{align*}
q^{N;1}[u_1^j, u_2^j] \rightarrow q^{N;1} \ \mathrm{in} \  C (0,T; L^2 (S^1)).
\end{align*}
Furthermore, it is easy to obtain that $q^{N;1} = q^{N;1} [u_1^*,u_2^*]$ by taking limits term by term in the weak formulation of (\ref{marginalequation}) and using uniqueness of solution of (\ref{marginalequation}). We omit the details of this standard argument. \\

We are now ready to prove (\ref{ineq3}). It holds
\begin{align*}
&\mathcal{J}_N (q^N [u_1^j,u_2^j], u_1^j, u_2^j) - \mathcal{J}_N (q^N [u_1^*,u_2^*],u_1^*, u_2^*) \\
= &\frac{\alpha_r}{2} \int_0^T \int_{S^1} \left((q^{N;1}[u_1^j, u_2^j] (t,\theta) - z(t,\theta))^2 - (q^{N;1}[u_1^*,u_2^*] (t,\theta) - z(t,\theta))^2 \right) \, d\theta dt \\ 
+ &\frac{\alpha_t}{2} \int_{S^1} \left((q^{N;1}[u_1^j, u_2^j] (T,\theta) - z(T,\theta))^2 - (q^{N;1}[u_1^*,u_2^*] (T,\theta) - z(T,\theta))^2 \right) \, d\theta \\ 
+ &\frac{1}{2} \int_0^T \int_{S^1} \left[ \beta_1 \left( (u_1^j (t,\theta))^2 - (u_1^* (t,\theta))^2 \right)  + \beta_2 \left( (u_2^j(t,\theta))^2 - (u_2^*(t,\theta))^2 \right) \right] \, d\theta dt \\
=: & I_1 + I_2 + I_3.
\end{align*}
We rewrite the terms $I_1$ and $I_2$
\begin{align*}
I_1 \leq &\frac{\alpha_r}{2} ||q^{N;1}[u_1^j, u_2^j] - q^{N;1}[u_1^*, u_2^*]||_{L^2(0,T; L^2 (S^1))} \\
&\left( ||q^{N;1}[u_1^j, u_2^j] - z||_{L^2(0,T; L^2 (S^1))} + ||q^{N;1}[u_1^*, u_2^*] - z||_{L^2(0,T; L^2 (S^1))} \right) 	, \\
I_2 \leq &\frac{\alpha_t}{2} ||q^{N;1}[u_1^j, u_2^j](T) - q^{N;1}[u_1^*, u_2^*](T)||_{L^2 (S^1)}  \\
&\left( ||q^{N;1}[u_1^j, u_2^j](T) - z(T)||_{L^2 (S^1)} + ||q^{N;1}[u_1^*, u_2^*](T) - z(T)||_{L^2 (S^1)} \right).
\end{align*}
Then we notice that from (\ref{bound1}) it follows
\begin{align*}
||q^{N;1}[u_1^j,u_2^j]||_{C(0,T; L^2(S^1))} \leq C
\end{align*}
for every $j \in \mathbb{N}$ and because of the convergence result in Step 1 it holds
\begin{align*}
||q^{N;1}[u_1^*,u_2^*]||_{C(0,T; L^2(S^1))} \leq ||q^{N;1}[u_1^*,u_2^*] - q^{N;1}[u_1^j,u_2^j]||_{C(0,T; L^2(S^1))} +||q^{N;1}[u_1^j,u_2^j]||_{C(0,T; L^2(S^1))} \leq C.
\end{align*}
Therefore the second bracket in the representation of $I_1$ and $I_2$ is bounded, and using strong convergence of $q^{N;1}[u_1^j, u_2^j]$ in $C(0,T; L^2(S^1))$, we obtain
\begin{align*}
I_1, I_2 \rightarrow 0, \ j \rightarrow \infty.
\end{align*}
For the term $I_3$ we obtain
\begin{align*}
\liminf_{j \rightarrow \infty} I_3 = \liminf_{j \rightarrow \infty} \Bigg( &\frac{\beta_1}{2} \left( ||u_1^j||_{L^2(0,T;L^2(S^1))}^2 - ||u_1^*||_{L^2(0,T;L^2(S^1))}^2 \right) \\
+ &\frac{\beta_2}{2} \left( ||u_2^j||_{L^2(0,T;L^2(S^1))}^2 - ||u_2^*||_{L^2(0,T;L^2(S^1))}^2   \right) \Bigg) \geq 0,
\end{align*}
since weak convergence of $(u_1^j)_{j \in \mathbb{N}}$ and $(u_2^j)_{j \in \mathbb{N}}$ in $L^\infty(0,T; W^{1,q}(S^1))$ automatically implies weak convergence in $L^2(0,T; L^2(S^1))$ to the same limit, and the $L^2(0,T; L^2(S^1))$ norm is weakly lower semicontinuous. Finally, we obtain
\begin{align*}
\lim_{j \rightarrow \infty} \left( \mathcal{J}_N (q^N [u_1^j,u_2^j], u_1^j, u_2^j) - \mathcal{J}_N (q^N [u_1^*,u_2^*],u_1^*, u_2^*) \right) 
\geq \liminf_{j \rightarrow \infty} I_1 + \liminf_{j \rightarrow \infty} I_2 + \liminf_{j \rightarrow \infty} I_3 \geq 0.
\end{align*}
\end{proof}

\begin{rem}
We notice that the bound for $||q^{N;1}[u_1^j, u_2^j]||_{L^\infty(0,T; H^1(S^1))}$ in Step 1 does not depend on $N$. We will need this fact in the proof of the main result.
\end{rem}

\section{Mean Field Limit for Fixed Control Functions}

We start by noticing that the Cauchy problem to the non-local PDE (\ref{mfpde}) has a unique solution, since the non-local term $\omega[q]$ is bounded. One sees directly that the equation preserves the total mass and, moreover, we can obtain an analogous result to Step 1 in the proof of Proposition \ref{JNminimizer} for the solution of the mean field PDE (\ref{mfpde}).

\begin{lemma}
The Cauchy problem corresponding to the non-local PDE (\ref{mfpde}) and initial data $q_0 \in L^2(S^1)$ has a unique weak solution in $C(0,T;L^2(S^1))$.
\end{lemma}

The proof of this result relies on the classical fixed point argument and $L^\infty (0,T; L^2(S^1))$ estimate.

\begin{lemma} \label{qconvergence}
Assume that the initial data is probability density $q_0 \in H^1(S^1)$. For a sequence $(u_1^N, u_2^N)_{N \in \mathbb{N}} \in \mathcal{U}^2$, there exists a weakly convergent (not relabeled) subsequence $(u_1^N, u_2^N) \rightharpoonup (\overline{u}_1, \overline{u}_2)$ in $\mathcal{U}^2$ such that the corresponding solution $q[u_1^N, u_2^N]$ of (\ref{mfpde}) converges to $q[\overline{u}_1, \overline{u}_2]$ in the following sense
\begin{align*}
q[u_1^N, u_2^N] \rightarrow q[\overline{u}_1, \overline{u}_2] \ \mathrm{in} \ C(0,T; L^2(S^1)).
\end{align*}
Moreover, it holds
\begin{align*}
||q[u_1^N, u_2^N]||_{C(0,T; H^1(S^1))} + ||q[\overline{u}_1, \overline{u}_2]||_{C(0,T; H^1(S^1))} \leq C
\end{align*}
for every $N \in \mathbb{N}$ and some uniform constant $C>0$.
\end{lemma}

The proof of this lemma follows easily by adjusting the arguments of the proof of Proposition \ref{JNminimizer}. \\

\begin{proposition} \label{l1convergence}
Let $q^N$ be the solution of (\ref{PDEqN}) and $q^{N;1}$ be its first marginal. Let \(q\) be the solution of (\ref{mfpde}), then we have
\begin{align*}
||q^{N;1} - q||_{L^\infty (0,T; L^1(S^1))} \leq \frac{C}{\sqrt{N}}
\end{align*}
for sufficiently big $N \in \mathbb{N}$.
\end{proposition}

The proof of this result follows the idea of relative entropy from \cite{MR3858403}. This method revolves around the control of the rescaled relative entropy
\begin{align*}
\mathcal{H} (q^N | q^{\otimes N}) (t) = \frac{1}{N} \int_{(S^1)^N} q^N (t, \theta_1, \dots, \theta_N) \log \frac{q^N (t, \theta_1, \dots, \theta_N)}{q^{\otimes N}(t, \theta_1, \dots, \theta_N)} d\theta_1 \dots d\theta_N,
\end{align*}
where $q^{\otimes N}$ is the tensorized law
\begin{align*}
q^{\otimes N} = \prod_{i=1}^N q(t,\theta_i)
\end{align*}
for the solution $q$ of (\ref{mfpde}). 
We derive the equation for $q^{\otimes N}$ and compute $\frac{d}{dt} \mathcal{H} (q^N | q^{\otimes N})$ to obtain the estimate for
\begin{align*}
\frac{C}{N} \int_{(S^1)^N}   \sum_{i=1}^N\left( \frac{u_2(\theta_i)}{N} \sum_{j=1}^N \sin(\theta_j - \theta_i - \alpha) - u_2(\theta_i) \int_{S^1} \sin(\theta' - \theta_i - \alpha) q(\theta') \, d\theta' \right)^2 q^N \, d\theta_1 \dots d\theta_N.
\end{align*}
Since we do not know the properties of $q^N$, we would much prefer having expectations with respect to the tensorized law $q^{\otimes N}$. Lemma \ref{jabinzhenfulemma} allows us this transition and we further bound the new term through a law of large number at the exponential scale given by Proposition \ref{jabinzhenfutheorem}. Finally, applying Grönwall's inequality and Csiszár-Kullback-Pinsker inequality finishes the proof. \\

\begin{proof}
In this proof we omit the dependence on the control functions and write $q^N[u_1,u_1] = q^N$ and $q[u_1,u_2] = q$. For reader's convenience, important results from \cite{MR3858403} used in the proof can be found in the appendix. \\

Direct computation shows that $q^{\otimes N}$ satisfies the equation
\begin{align*}
\partial_t q^{\otimes N} - \sum_{i=1}^N D \partial_{\theta_i}^2 q^{\otimes N} + \sum_{i=1}^N \partial_{\theta_i} \left( q^{\otimes N} \left( u_1(t,\theta_i) + u_2(t,\theta_i) \int_{S^1} \sin (\theta' - \theta_i - \alpha) q(\theta') \, d\theta' \right) \right) = 0.
\end{align*}
Using this and the equation for $q^N$, integrating by parts and using Hölder's and Young's inequalities, we get
\begin{align*}
\begin{split}
&\frac{d}{dt} \mathcal{H} (q^N | q^{\otimes N}) \\
& = \frac{1}{N} \int_{(S^1)^N} \left( \partial_t q^N -  \partial_t q^{\otimes N} \frac{q^N}{q^{\otimes N}} + \log \left( \frac{q^N}{q^{\otimes N}} \right) \partial_t q^N \right) \, d\theta_1 \dots d\theta_N \\
& = \frac{1}{N} \int_{(S^1)^N} \sum_{i=1}^N \left( D \partial_{\theta_i} q^{\otimes N}  - q^{\otimes N} \left( u_1(\theta_i)  + u_2(\theta_i) \int_{S^1} \sin(\theta' - \theta_i - \alpha) q(\theta') \, d\theta' \right) \right) \partial_{\theta_i} \frac{q^N}{q^{\otimes N}}  \, d\theta_1 \dots d\theta_N \\
& - \frac{1}{N} \int_{(S^1)^N} \sum_{i=1}^N \left( D \partial_{\theta_i} q^N  - q^N \left( u_1(\theta_i)  + \frac{u_2(\theta_i)}{N} \sum_{j=1}^N \sin(\theta_j - \theta_i - \alpha) \right) \right) \partial_{\theta_i} \log \left( \frac{q^N}{q^{\otimes N}} \right)  \, d\theta_1 \dots d\theta_N \\
& \leq -\frac{D}{N} \int_{(S^1)^N} \sum_{i=1}^N \left\vert \partial_{\theta_i} \log \left( \frac{q^N}{q^{\otimes N}} \right) \right\vert^2 q^N \, d\theta_1 \dots d\theta_N + \frac{D}{2N} \sum_{i=1}^N\int_{(S^1)^N} q^N \left\vert \partial_{\theta_i} \log \left( \frac{q^N}{q^{\otimes N}} \right) \right\vert^2  \, d\theta_1 \dots d\theta_N  \\
& +\frac{1}{2DN} \sum_{i=1}^N\int_{(S^1)^N}  q^N \left( \frac{u_2(\theta_i)}{N} \sum_{j=1}^N \sin(\theta_j - \theta_i - \alpha) - u_2(\theta_i) \int_{S^1} \sin(\theta' - \theta_i - \alpha) q(\theta') \, d\theta' \right)^2 \, d\theta_1 \dots d\theta_N  \\
& \leq -\frac{D}{2N} \int_{(S^1)^N}  \sum_{i=1}^N \left\vert \partial_{\theta_i} \log \left( \frac{q^N}{q^{\otimes N}} \right) \right\vert^2  q^N \, d\theta_1 \dots d\theta_N \\
& + \frac{1}{2DN} \int_{(S^1)^N}   \sum_{i=1}^N\left( \frac{u_2(\theta_i)}{N} \sum_{j=1}^N \sin(\theta_j - \theta_i - \alpha) - u_2(\theta_i) \int_{S^1} \sin(\theta' - \theta_i - \alpha) q(\theta') \, d\theta' \right)^2 q^N \, d\theta_1 \dots d\theta_N.
\end{split}
\end{align*}
We now analyze the second term that we rewrite in the following way
\begin{align*}
&\frac{1}{2DN} \int_{(S^1)^N}   \sum_{i=1}^N\left( \frac{u_2(\theta_i)}{N} \sum_{j=1}^N \sin(\theta_j - \theta_i - \alpha) - u_2(\theta_i) \int_{S^1} \sin(\theta' - \theta_i - \alpha) q(\theta') \, d\theta' \right)^2 q^N \, d\theta_1 \dots d\theta_N \\
\leq & \frac{8e^2 \tilde{M}^2}{D N} \int_{(S^1)^N}   \sum_{i=1}^N \frac{1}{16e^2} \left( \frac{1}{N} \sum_{j=1}^N \sin(\theta_j - \theta_i - \alpha) -  \int_{S^1} \sin(\theta' - \theta_i - \alpha) q(\theta') \, d\theta' \right)^2 q^N \, d\theta_1 \dots d\theta_N.
\end{align*}
We use the strategy from \cite{MR3858403} and apply Lemma \ref{jabinzhenfulemma} to the function
\begin{align*}
\Phi_i = \frac{1}{16e^2} \left( \frac{1}{N} \sum_{j=1}^N \sin(\theta_j - \theta_i - \alpha) - \int_{S^1} \sin(\theta' - \theta_i - \alpha) q(\theta') \, d\theta' \right)^2
\end{align*}
and obtain
\begin{align*}
\frac{1}{N} \sum_{i=1}^N \int_{(S^1)^N} \Phi_i q^N \, d\theta_1 \dots d\theta_N \leq \mathcal{H} (q^N\vert q^{\otimes N}) + \frac{1}{N^2} \sum_{i=1}^N \log \int_{(S^1)^N} q^{\otimes N} e^{N\Phi_i} \, d\theta_1 \dots d\theta_N.
\end{align*}
We define
\begin{align*}
\psi (\theta_i, \theta_j) = \frac{1}{4e} \sin (\theta_j - \theta_i - \alpha) -  \frac{1}{4e} \int_{S^1} \sin (\theta ' - \theta_i - \alpha) q(\theta ') \,d\theta'
\end{align*}
and hence we obtain
\begin{align*}
\exp(N\Phi_i) = \exp \left( \frac{1}{\sqrt{N}} \sum_{j=1}^N  \psi (\theta_i, \theta_j) \right)^2 = \exp \left( \frac{1}{N} \sum_{j_1,j_2=1}^N \psi (\theta_i, \theta_{j_1}) \psi (\theta_i, \theta_{j_2}) \right).
\end{align*}
We remark that $\psi$ has vanishing expectation with respect to $q$ and it holds
\begin{align*}
||\psi||_{L^\infty} \leq \frac{1}{2e}.
\end{align*}
So we apply Proposition \ref{jabinzhenfutheorem} and obtain
\begin{align*}
\int_{(S^1)^N} q^{\otimes N} \exp \left( \frac{1}{N} \sum_{j_1,j_2=1}^N \psi (\theta_i, \theta_{j_1}) \psi (\theta_i, \theta_{j_2}) \right) \, d\theta_1 \dots d\theta_N \leq C.
\end{align*}
Therefore
\begin{align*}
\frac{1}{N^2} \sum_{i=1}^N \log \int_{(S^1)^N} q^{\otimes N} e^{N\Phi_i} \, d\theta_1 \dots d\theta_N \leq \frac{C}{N}
\end{align*}
and finally
\begin{align*}
&\frac{d}{dt} \mathcal{H} (q^N | q^{\otimes N}) + \frac{D}{2N} \int_{(S^1)^N}  \sum_{i=1}^N \left\vert \partial_{\theta_i} \log \left( \frac{q^N}{q^{\otimes N}} \right) \right\vert^2  q^N \, d\theta_1 \dots d\theta_N  \leq \frac{8 e^2 \tilde{M}^2}{D}\mathcal{H} (q^N\vert q^{\otimes N}) + \frac{8e^2 \tilde{M}^2 C}{DN}.
\end{align*}
Applying Grönwall's inequality we obtain
\begin{align*}
\mathcal{H} (q^N| q^{\otimes N})(t) \leq \frac{8e^2 \tilde{M}^2 CT}{DN} \exp \left( \frac{8e^2 \tilde{M}^2 T}{D} \right), \ \forall t \in [0,T].
\end{align*}
Finally using Csiszár-Kullback-Pinsker inequality (see \cite{MR2459454}), we conclude
\begin{align*}
||q^{N;1} - q ||^2_{L^\infty (0,T; L^1(S^1))} \leq 2 ||\mathcal{H} (q^N| q^{\otimes N})(t)||_{L^\infty (0,T)} \leq \frac{C}{N}.
\end{align*}
This proves the claim.
\end{proof}

\section{Proof of the Main Result}

This section is devoted to the proof of the main result, Theorem \ref{mainresult}. We start by proving the following convergence result for the cost functionals $\mathcal{J}_N$ and $\mathcal{J}$.

\begin{lemma} \label{minlimit}
Assume that the initial data is probability density $q_0 \in H^1(S^1)$ and assume $(\hat{u}^N_1,\hat{u}^N_2)_{N \in \mathbb{N}}$ is a sequence in $\mathcal{U}^2$.
Then it holds
\begin{align*}
\lim_{N \rightarrow \infty} \left( \mathcal{J}_N (q^N [\hat{u}^N_1,\hat{u}^N_2],\hat{u}^N_1,\hat{u}^N_2) - \mathcal{J} (q [\hat{u}^N_1, \hat{u}^N_2], \hat{u}^N_1, \hat{u}^N_2) \right) = 0.
\end{align*}
\end{lemma}

\begin{proof}
According to the definitions of $\mathcal{J}_N$ and $\mathcal{J}$, we have
\begin{align*}
&\left\vert \mathcal{J}_N (q^N [\hat{u}^N_1,\hat{u}^N_2],\hat{u}^N_1,\hat{u}^N_2) - \mathcal{J} (q [\hat{u}^N_1, \hat{u}^N_2], \hat{u}^N_1, \hat{u}^N_2) \right\vert \\ 
\leq &\frac{\alpha_r}{2} \int_0^T \int_{S^1} \vert (q^{N;1}[\hat{u}^N_1,\hat{u}^N_2](t,\theta) - z(t,\theta))^2 - (q[\hat{u}^N_1,\hat{u}^N_2](t,\theta) - z(t,\theta))^2 \vert \, d\theta dt \\
&+ \frac{\alpha_t}{2} \int_{S^1} \vert (q^{N;1}[\hat{u}^N_1,\hat{u}^N_2](T,\theta) - z(T,\theta))^2 - (q[\hat{u}^N_1,\hat{u}^N_2](T,\theta) - z(T,\theta))^2 \vert \, d\theta  \\
=: & J_1 + J_2.
\end{align*}
According to the proofs of Proposition \ref{JNminimizer} and Lemma \ref{qconvergence}, we can bound 
\begin{align} \label{b2}
\begin{split}
&||q^{N;1} [\hat{u}_1^N, \hat{u}_2^N] ||_{L^\infty (0,T; H^1(S^1))} \leq C, \\
& ||q [\hat{u}_1^N, \hat{u}_2^N] ||_{L^\infty (0,T; H^1(S^1))} \leq C,
\end{split}
\end{align}
where the constant $C$ does not depend on $N$. Since it holds 
\begin{align*}
J_1 \leq & \frac{\alpha_r}{2} ||q^{N;1}[\hat{u}^N_1,\hat{u}^N_2]- q[\hat{u}^N_1,\hat{u}^N_2]||_{L^\infty (0,T; L^1 (S^1))} ||q^{N;1}[\hat{u}^N_1,\hat{u}^N_2] + q[\hat{u}^N_1,\hat{u}^N_2] - 2z||_{L^1 (0,T; L^\infty (S^1))}, \\
J_2 \leq & \frac{\alpha_t}{2} ||q^{N;1}[\hat{u}^N_1,\hat{u}^N_2]- q[\hat{u}^N_1,\hat{u}^N_2]||_{L^\infty (0,T; L^1 (S^1))} ||q^{N;1}[\hat{u}^N_1,\hat{u}^N_2] + q[\hat{u}^N_1,\hat{u}^N_2] - 2z||_{L^\infty (0,T; L^\infty (S^1))},
\end{align*}
we use (\ref{b2}) and Proposition \ref{l1convergence} and obtain
\begin{align*}
J_1\rightarrow 0 \ \mathrm{and} \ J_2 \rightarrow 0 \ \mathrm{for} \  N \rightarrow 0.
\end{align*}
\end{proof}

We are now ready to finish the proof of Theorem \ref{mainresult}. \\

\textit{Proof of Theorem \ref{mainresult}.}
For any fixed $N$, we know there exist minimizers of $\mathcal{J}_N (q^N [u_1,u_2],u_1,u_2)$ by Proposition \ref{JNminimizer}. We denote these minimizers by $u_1^N, u_2^N$, i.e.
\begin{align*}
\mathcal{J}_N (q^N [u_1^N,u_2^N],u_1^N,u_2^N) = \min_{u_1,u_2 \in \mathcal{U}} \mathcal{J}_N (q^N [u_1,u_2],u_1,u_2).
\end{align*}
According to Lemma \ref{qconvergence}, there exists a weak limit of a subsequence $(u_1^{N_k},u_2^{N_k})$ of $(u_1^N,u_2^N)$ in $\mathcal{U}^2$, which we denote by $(\overline{u}_1,\overline{u}_2)$, such that
\begin{align} \label{convq}
q[u_1^{N_k}, u_2^{N_k}] \rightarrow q[\overline{u}_1, \overline{u}_2] \ \mathrm{in} \ C(0,T; L^2(S^1)).
\end{align} 
We start by establishing
\begin{align} \label{lsc}
\liminf_{k \rightarrow \infty} \mathcal{J}_{N_k} (q^{N_k} [u_1^{N_k},u_2^{N_k}],u_1^{N_k},u_2^{N_k}) \geq \mathcal{J} (q[\overline{u}_1,\overline{u}_2],\overline{u}_1,\overline{u}_2).
\end{align}
We observe that it is enough to prove
\begin{align*}
\liminf_{k \rightarrow \infty} \left( \mathcal{J}_{N_k} (q^{N_k} [u_1^{N_k},u_2^{N_k} ],u_1^{N_k},u_2^{N_k}) - \mathcal{J} (q[u_1^{N_k}, u_2^{N_k}],u_1^{N_k}, u_2^{N_k}) \right) \\
+ \liminf_{k \rightarrow \infty}\mathcal{J} (q[u_1^{N_k}, u_2^{N_k}],u_1^{N_k}, u_2^{N_k})\geq \mathcal{J} (q[\overline{u}_1, \overline{u}_2],\overline{u}_1, \overline{u}_2).
\end{align*}
We first apply Lemma \ref{minlimit} and see that the first term on the left-hand side vanishes. Then
\begin{align*}
&\mathcal{J} (q [u_1^{N_k},u_2^{N_k}], u_1^{N_k}, u_2^{N_k}) - \mathcal{J}(q[\overline{u}_1,\overline{u}_2],\overline{u}_1, \overline{u}_2) \\
= &\frac{\alpha_r}{2} \int_0^T \int_{S^1} \left( q[u_1^{N_k},u_2^{N_k}] (t,\theta) - q [\overline{u}_1, \overline{u}_2](t,\theta) \right) \left( q[u_1^{N_k},u_2^{N_k}] (t,\theta) + q [\overline{u}_1, \overline{u}_2](t,\theta) - 2z(t,\theta) \right) \, d\theta dt \\ 
+ &\frac{\alpha_t}{2} \int_{S^1} \left( q [u_1^{N_k},u_2^{N_k}](T,\theta) - q[\overline{u}_1,\overline{u}_2] (T,\theta) \right) \left( q [u_1^{N_k},u_2^{N_k}](T,\theta) + q[\overline{u}_1,\overline{u}_2] (T,\theta) - 2z(T,\theta) \right) \, d\theta \\ 
+ &\frac{1}{2} \int_0^T \int_{S^1} \left[ \beta_1 \left( \left( u_1^{N_k} (t,\theta)\right)^2 - \left( \overline{u}_1 (t,\theta) \right)^2 \right)  + \beta_2 \left( \left( u_2^{N_k}(t,\theta) \right)^2 - \left( \overline{u}_2(t,\theta) \right)^2 \right) \right] \, d\theta dt \\
=: & K_1 + K_2 + K_3.
\end{align*}
We observe
\begin{align*}
K_1 \leq &\frac{\alpha_r}{2} ||q[u_1^{N_k}, u_2^{N_k}] - q[\overline{u}_1, \overline{u}_2]||_{L^2(0,T, L^2 (S^1))} \\
&\left(||q[u_1^{N_k}, u_2^{N_k}] - z||_{L^2(0,T, L^2 (S^1))} + ||q[\overline{u}_1, \overline{u}_2] - z||_{L^2(0,T, L^2 (S^1))} \right) \\
K_2  \leq &\frac{\alpha_t}{2} ||q[u_1^{N_k}, u_2^{N_k}](T) - q[\overline{u}_1,\overline{u}_2](T)||_{L^2 (S^1)}  \\
&\left( ||q[u_1^{N_k}, u_2^{N_k}](T) - z(T)||_{L^2 (S^1)} + ||q[\overline{u}_1,\overline{u}_2](T) - z(T)||_{L^2 (S^1)} \right).
\end{align*}
Using (\ref{convq}) and Lemma \ref{qconvergence}, it follows directly
\begin{align*}
K_1, K_2 \rightarrow 0, \ k \rightarrow \infty.
\end{align*}
For the term $K_3$ we obtain
\begin{align*}
\liminf_{k \rightarrow \infty} K_3 = &\liminf_{k \rightarrow \infty} \Bigg( \frac{\beta_1}{2} \left( ||u_1^{N_k}||_{L^2(0,T;L^2(S^1))}^2 - ||\overline{u}_1||_{L^2(0,T;L^2(S^1))}^2 \right) \\
+ &\frac{\beta_2}{2} \left( ||u_2^{N_k}||_{L^2(0,T;L^2(S^1))}^2 - ||\overline{u}_2||_{L^2(0,T;L^2(S^1))}^2  \right) \Bigg) \geq 0
\end{align*}
by the lower semicontinuity of the $L^2(0,T; L^2(S^1))$ norm. This proves (\ref{lsc}). \\

Now we notice that for any $(\hat{u}_1, \hat{u}_2)\in \mathcal{U}$, using Lemma \ref{minlimit}, we can obtain
\begin{align*}
\mathcal{J}(q[\hat{u}_1, \hat{u}_2],\hat{u}_1, \hat{u}_2) =& \lim_{N \rightarrow \infty} \mathcal{J}_N (q^N [\hat{u}_1, \hat{u}_2],\hat{u}_1, \hat{u}_2) =  \lim_{k \rightarrow \infty} \mathcal{J}_{N_k} (q^{N_k} [\hat{u}_1, \hat{u}_2],\hat{u}_1, \hat{u}_2) \\
\geq & \liminf_{k \rightarrow \infty} \min_{u_1, u_2 \in \mathcal{U}} \mathcal{J}_{N_k} (q^{N_k} [u_1,u_2],u_1,u_2) \geq  \mathcal{J} (q[\overline{u}_1, \overline{u}_2], \overline{u}_1, \overline{u}_2).
\end{align*}
This shows
\begin{align*}
\mathcal{J}(q[\overline{u}_1, \overline{u}_2], \overline{u}_1, \overline{u}_2) = \min_{u_1, u_2 \in \mathcal{U}} \mathcal{J} (q[u_1,u_2],u_1,u_2).
\end{align*}
\qed

\appendix

\section{Appendix}

\begin{lemma}[Lemma 1 in \cite{MR3858403}] \label{jabinzhenfulemma}
For any two probability densities $\rho_N$ and $\overline{\rho}_N$ on $(S^1)^N$, and any $\Phi \in L^\infty ((S^1)^N)$, one has that 
\begin{align*}
\int_{(S^1)^N} \Phi \rho_N \, dX^N \leq \mathcal{H} (\rho_N | \overline{\rho}_N) + \frac{1}{N} \log \int_{(S^1)^N} \overline{\rho}_N e^{N\Phi} \, dX^N,
\end{align*}
where $X^N = (x_1, \dots, x_N) \in (S^1)^N$.
\end{lemma}

\begin{proposition}[Theorem 3 in \cite{MR3858403}] \label{jabinzhenfutheorem}
Consider any $\overline{\rho} \in L^1 (S^1)$ with $\overline{\rho}\geq 0$ and $\int_{S^1} \overline{\rho}(x)\, dx =1$. Assume that a scalar function $\psi \in L^\infty$ with $||\psi||_{L^\infty} < \frac{1}{2e}$, and that for any fixed $z$, $\int_{S^1} \psi (z,x) \overline{\rho}(x) \, dx = 0$, then
\begin{align*}
\int_{(S^1)^N} \overline{\rho}_N \exp \left( \frac{1}{N} \sum_{j_1,j_2 =1}^N \psi (x_1,x_{j_1}) \psi (x_1, x_{j_2}) \right) \, dX^N \leq C = 2 \left(1 + \frac{10\alpha}{(1-\alpha)^3} + \frac{\beta}{1-\beta} \right),
\end{align*}
where $\overline{\rho}_N (t,X^N) = \Pi_{i=1}^N \overline{\rho} (t,x_i)$ and
\begin{align*}
\alpha = (e||\psi||_{L^\infty})^4 <1, \ \ \beta = \left( \sqrt{2e} ||\psi||_{L^\infty} \right)^4 <1.
\end{align*}
\end{proposition}

\section*{Acknowledgments}

The research of Li Chen and Valeriia Zhidkova was partially supported by the German Research Foundation (No. CH 955/8-1). The research of Yucheng Wang was partially supported by the National Natural Science Foundation of China (No.12271357).

\nocite{*}
\bibliography{Lit}

\end{document}